\documentclass[a4paper,10.9pt]{amsart}

\DeclareRobustCommand{\SkipTocEntry}[5]{}

\usepackage{shadow}

\usepackage{accents}

\usepackage{marvosym}

\usepackage[T1]{fontenc}
\usepackage{empheq}
\usepackage{relsize}
\usepackage{amsmath}

\usepackage{bm}

\usepackage{upgreek}
\usepackage{ esint }
\usepackage{color}
\usepackage{comment}
\usepackage{amssymb}
\usepackage{amsfonts}
\usepackage{graphicx}

\usepackage{amscd}

\usepackage{slashed}
\usepackage{pdflscape}
\usepackage{tikz}
\usepackage{enumerate}
\usepackage{multirow}
\usepackage{stmaryrd}
\usepackage{cancel}

\usepackage{scalerel,stackengine}

\usepackage{ifthen}

\usepackage{bbold}
\usepackage[linkcolor=blue]{hyperref}
\usepackage{mathrsfs}
\usepackage[colorinlistoftodos]{todonotes}

\definecolor{blue}{rgb}{.255,.41,.884} 
\definecolor{red}{rgb}{1, 0, 0} 
\definecolor{green}{rgb}{.196,.804,.196} 
\definecolor{yellow}{rgb}{1,.648,0} 
\definecolor{pink}{rgb}{1,0.5,0.5}

\newtheorem{theorem}{Theorem}[section]
\newtheorem{lemma}[theorem]{Lemma}
  
\newtheorem{proposition}[theorem]{Proposition}
\newtheorem{corollary}[theorem]{Corollary}

\theoremstyle{definition}
\newtheorem{definition}[theorem]{Definition}
\newtheorem{example}[theorem]{Example}

\theoremstyle{remark}
\newtheorem{remark}[theorem]{Remark}

\usepackage{pdfsync}

\usepackage{graphicx} 
\usepackage{amsmath} 
\usepackage{amsfonts}
\usepackage{amssymb}

\newcommand{\be}{\begin{equation}}

\newcommand{\ee}{\end{equation}}

\newcommand{\II}{{\rm  I\hspace{-.2mm}I}}
\newcommand{\IIo}{\hspace{0.4mm}\mathring{\rm{ I\hspace{-.2mm} I}}{\hspace{.0mm}}}

\newcommand{\IIIo}{{\mathring{{\bf\rm I\hspace{-.2mm} I \hspace{-.2mm} I}}{\hspace{.2mm}}}{}}
\newcommand{\IVo}{{\mathring{{\bf\rm I\hspace{-.2mm} V}}{\hspace{.2mm}}}{}}
\newcommand{\Vo}{{\mathring{{\bf\rm V}}}{}}

\newcommand{\otop}{\mathring{\top}}

\newcommand{\ba}{\begin{array}}

\newcommand{\ea}{\end{array}}

\newcommand{\beq}{\begin{eqnarray}}

\newcommand{\eeq}{\end{eqnarray}}

\newtheorem{lm}{lemma}

\newtheorem{thee}{theorem}

\newtheorem{proo}{proposition}

\newtheorem{co}{corollary}

\newtheorem{rem}{remark}

\newtheorem{deff}{definition}

\newcommand{\bd}{\begin{deff}}

\newcommand{\ed}{\end{deff}}

\newcommand{\bl}{\begin{lm}}

\newcommand{\el}{\end{lm}}

\newcommand{\bp}{\begin{proo}}

\newcommand{\ep}{\end{proo}}

\newcommand{\bt}{\begin{thee}}

\newcommand{\et}{\end{thee}}

\newcommand{\bc}{\begin{co}}

\newcommand{\ec}{\end{co}}

\newcommand{\brm}{\begin{rem}}

\newcommand{\erm}{\end{rem}}

\usepackage{t1enc}

\newcommand{\bS}{\mathbb{S}}

\newcommand{\newc}{\newcommand}

\let\ccdot.

\newc{\aR}{\mbox{\boldmath{$ R$}}}
\newc{\aS}{\mbox{\boldmath{$ S$}}}
\newc{\aT}{\mbox{\boldmath{$ T$}}}
\newc{\aW}{\mbox{\boldmath{$ W$}}}

\newc{\aD}{\mbox{\boldmath{$ D$}}\hspace{-.2mm}}

\renewcommand{\colon}{\scalebox{1.2}{:}}

\newc{\aK}{\mbox{\boldmath{$ K$}}}
\newc{\aL}{\mbox{\boldmath{$ L$}}}

\newcommand{\bT}{{\Bbb T}}

\usepackage{amssymb}
\usepackage{amscd}

\newcommand{\Sc}{{\it Sc}}

\newcommand{\nn}[1]{(\ref{#1})}

\newc{\obstrn}[2]{B^{#1}_{#2}}

\newcommand{\rpl}                         
{\mbox{$
\begin{picture}(12.7,8)(-.5,-1)
\put(0,0.2){$+$}
\put(4.2,2.8){\oval(8,8)[r]}
\end{picture}$}}

\newcommand{\lpl}                         
{\mbox{$
\begin{picture}(12.7,8)(-.5,-1)
\put(2,0.2){$+$}
\put(6.2,2.8){\oval(8,8)[l]}
\end{picture}$}}

\newc{\tensor}[1]{#1}
\newc{\Mvariable}[1]{\mbox{#1}}
\newc{\down}[1]{{}_{#1}}
\newc{\up}[1]{{}^{#1}}

\newc{\JulyStrut}{\rule{0mm}{6mm}}
\newc{\midtenPan}{\mbox{\sf S}}
\newc{\midten}{\mbox{\sf T}}
\newc{\midtenEi}{\mbox{\sf U}}
\newc{\ATen}{\mbox{\sf E}}
\newc{\BTen}{\mbox{\sf F}}
\newc{\CTen}{\mbox{\sf G}}

\def\sideremark#1{\ifvmode\leavevmode\fi\vadjust{\vbox to0pt{\vss
 \hbox to 0pt{\hskip\hsize\hskip1em
 \vbox{\hsize2cm\tiny\raggedright\pretolerance10000
  \noindent #1\hfill}\hss}\vbox to8pt{\vfil}\vss}}}

\numberwithin{equation}{section}

\newcommand{\hh}{{\hspace{.3mm}}}

\renewcommand\colon{\scalebox{1.3}{$:$}}

\newcommand{\cc}{\boldsymbol{c}}

\newcommand{\pdot}{{\boldsymbol{\cdot}}}

\newcommand{\sss}{\scriptscriptstyle}

\renewcommand\geq{\geqslant}
\renewcommand\leq{\leqslant}

\stackMath
\newcommand\reallywidehat[1]{%
\savestack{\tmpbox}{\stretchto{%
  \scaleto{%
    \scalerel*[\widthof{\ensuremath{#1}}]{\kern-.6pt\bigwedge\kern-.6pt}%
    {\rule[-\textheight/2]{1ex}{\textheight}}
  }{\textheight}%
}{0.5ex}}%
\stackon[1pt]{#1}{\tmpbox}%
}

\begin{document}


\renewcommand{\today}{}
\title{
{Toward a Classification of Conformal Hypersurface Invariants
}}
%
%
%

\author{ Samuel Blitz${}^\flat$}

\address{${}^\flat$
 Department of Mathematics and Statistics \\
 Masaryk University\\
 Building 08, Kotl\'a\v{r}sk\'a 2 \\
 Brno, CZ 61137} 
   \email{blitz@math.muni.cz}
 
\vspace{10pt}

\renewcommand{\arraystretch}{1}

\begin{abstract}


Hypersurfaces embedded in conformal manifolds appear frequently as boundary data in boundary-value problems in cosmology and string theory. Viewed as the non-null conformal infinity of a spacetime, we consider hypersurfaces embedded in a Riemannian (or Lorentzian) conformal manifold. We construct a finite and minimal family of hypersurface tensors---the curvatures intrinsic to the hypersurface and the so-called ``conformal fundamental forms''---that can be used to construct natural conformal invariants of the hypersurface embedding up to a fixed order in hypersurface-orthogonal derivatives of the bulk metric. We thus show that these conformal fundamental forms capture the extrinsic embedding data of a conformal infinity in a spacetime.


\vspace{2cm}

\end{abstract}


\maketitle

\pagestyle{myheadings} \markboth{S. Blitz}{Conformal Hypersurface Invariants}



\newcommand{\balpha}{{\bm \alpha}}
\newcommand{\balphas}{{\scalebox{.76}{${\bm \alpha}$}}}
\newcommand{\bnu}{{\bm \nu}}
\newcommand{\blambda}{{\bm \lambda}}
\newcommand{\bnus}{{\scalebox{.76}{${\bm \nu}$}}}
\newcommand{\bnuss}{\hh\hh\!{\scalebox{.56}{${\bm \nu}$}}}

\newcommand{\bmu}{{\bm \mu}}
\newcommand{\bmus}{{\scalebox{.76}{${\bm \mu}$}}}
\newcommand{\bmuss}{\hh\hh\!{\scalebox{.56}{${\bm \mu}$}}}

\newcommand{\btau}{{\bm \tau}}
\newcommand{\btaus}{{\scalebox{.76}{${\bm \tau}$}}}
\newcommand{\btauss}{\hh\hh\!{\scalebox{.56}{${\bm \tau}$}}}

\newcommand{\bsigma}{{\bm \sigma}}
\newcommand{\bsigmas}{{{\scalebox{.8}{${\bm \sigma}$}}}}
\newcommand{\bbeta}{{\bm \beta}}
\newcommand{\bbetas}{{\scalebox{.65}{${\bm \beta}$}}}

\renewcommand{\bS}{{\bm {\mathcal S}}}
\newcommand{\bB}{{\bm {\mathcal B}}}
\renewcommand{\bT}{{\bm {\mathcal T}}}
\newcommand{\bM}{{\bm {\mathcal M}}}

\newcommand{\go}{{\mathring{g}}}
\newcommand{\nuo}{{\mathring{\nu}}}
\newcommand{\alphao}{{\mathring{\alpha}}}

\newcommand{\Ell}{\mathscr{L}}
\newcommand{\density}[1]{[g\, ;\, #1]}

\renewcommand{\Dot}{{\scalebox{2}{$\cdot$}}}

\newcommand{\PanE}{P_{4}^{\sss\Sigma\hookrightarrow M}}
\newcommand\eqSig{ \mathrel{\overset{\makebox[0pt]{\mbox{\normalfont\tiny\sffamily~$\Sigma$}}}{=}} }
\renewcommand\eqSig{\mathrel{\stackrel{\Sigma\hh}{=}} }
\newcommand\eqtau{\mathrel{\overset{\makebox[0pt]{\mbox{\normalfont\tiny\sffamily~$\tau$}}}{=}}}
\newcommand{\hd }{\hat{D}}
\newcommand{\hdb}{\hat{\bar{D}}}
\newcommand{\Two}{{{{\bf\rm I\hspace{-.2mm} I}}{\hspace{.2mm}}}{}}
\newcommand{\TwoN}{{\mathring{{\bf\rm I\hspace{-.2mm} I}}{\hspace{.2mm}}}{}}
\newcommand{\Fn}{\mathring{\mathcal{F}}}
\newcommand{\csdot}{\hspace{-0.75mm} \cdot \hspace{-0.75mm}}
\newcommand{\IdD}{(I \csdot \hd)}
\newcommand{\Kd}{\dot{K}}
\newcommand{\Kdd}{\ddot{K}}
\newcommand{\Kddd}{\dddot{K}}

 \newcommand{\bdot }{\mathop{\lower0.33ex\hbox{\LARGE$\cdot$}}}

\definecolor{ao}{rgb}{0.0,0.0,1.0}
\definecolor{forest}{rgb}{0.0,0.3,0.0}
\definecolor{red}{rgb}{0.8, 0.0, 0.0}

\newcommand{\APE}[1]{{\rm APE}_{#1}}
\newcommand{\PE}{{\rm PE}}
\newcommand{\FF}[1]{\mathring{\underline{\overline{\rm{#1}}}}}
\newcommand{\ltots}[1]{{\rm ltots}_{#1}}
\newcommand{\tsim}[1]{\stackrel{#1}{\sim}}

\section{Introduction}
In Lorentzian spacetimes with non-zero cosmological constant, the conformal infinity is either spacelike ($\Lambda > 0$) or timelike ($\Lambda < 0$). The structure of conformal infinities of such spacetimes have become a topic of wide interest in both the mathematical and physics communities, providing tools for the AdS/CFT correspondence~\cite{Maldacena,GrWi,GalGraLin}, the conformal cyclic cosmology paradigm~\cite{tod,nurowski,GoKop}, the study of generalized Willmore invariants~\cite{hypersurface_old,Olanipekun,BGW2,Blitz1}, among many others. Thus, rather than studying the physical spacetime $(\hat{M}, \hat{g})$ itself (which does not include its conformal infinity), it is instructive to study its conformal extension to a so-called ``unphysical spacetime'' $(M,g^o)$ containing its conformal boundary. Here, $M$ is a compact manifold with boundary $\partial M$, the interior $M_+$ is diffeomorphic to $\hat{M}$, and $g^o$ is singular along $\Sigma := \partial M$.

Locally, we can always express the boundary of a compact manifold as the zero locus of a \textit{defining function}, \textit{i.e.} a function $s$ such that for all $p \in \Sigma$, we have that $s(p) = 0 \neq ds(p)$. So, according to the Gau\ss\ lemma, we can express this singular metric in terms of normal coordinates, one of which is our defining function. Normalizing dimensions, for $\Lambda > 0$ we can write this singular metric as
$$
g^o = \frac{- d s^2 + h_{s}}{s^2}\,,
$$
where the coordinate vector field $\partial_s$ is inward pointing and timelike and $h_{s}$ is a 1-parameter family of Riemannian metrics on $\Sigma$. Similarly, for $\Lambda < 0$, we have that
$$
g^o = \frac{ds^2 + h_s}{s^2}\,,
$$
where $\partial_s$ is inward pointing and spacelike and $h_s$ a 1-parameter family of Lorentzian metrics on $\Sigma$.

Observe that the choice of defining function was arbitrary: indeed, if $s$ is a defining function for $\Sigma$ and $0 < \Omega \in C^\infty M$, clearly $\Omega s$ is also a defining function for $\Sigma$. Thus, associated to an unphysical Lorentzian spacetime $(M,g^o)$ with non-zero cosmological constant, there is associated a \textit{conformal hypersurface embedding} $\Sigma \hookrightarrow (M,\cc)$, where $\cc = [\Omega^2 s^2 g^o]$ is an equivalence class of metrics determined by positive smooth functions $\Omega$ on $M$. Put another way, for any $g,\tilde{g} \in \cc$, there exists $\Omega \in C^\infty_+ M$ such that $\tilde{g} = \Omega^2 g$; a manifold equipped with such a conformal class of metrics $(M,\cc)$ is called a \textit{conformal manifold}. By virtue of this association, studying the conformal infinities of Lorentzian spacetimes can be viewed equivalently as studying conformal hypersurface embeddings. Furthermore, Anderson~\cite{Anderson2} proved that (at least locally) there is a bijective correspondence between conformal hypersurface embeddings with Lorentzian signature and Riemannian signature via Wick rotation. To that end, we consider hypersurfaces embedded in conformal manifolds with all-plus signature for the remainder of this article, noting that only straightforward sign changes occur when Wick-rotating from these manifolds to those with Lorentzian signatures.

We now recall useful notions of standard Riemannian hypersurface geometry. Given a smooth compact Riemannian manifold $(M,g)$ with boundary $\Sigma := \partial M$, one can locally characterize the embedding $\Sigma \hookrightarrow (M,g)$ in terms of classically-known \textit{hypersurface invariants} of $\Sigma$, where $\Sigma$ is viewed as a hypersurface in $M$. These include the unit conormal $\hat{n}$ to $\Sigma$, the second fundamental form $\II$, and the mean curvature $H$; beyond these, various projections to $\Sigma$ of curvatures (and their derivatives) of $(M,g)$ can further elaborate on the embedding in a diffeomorphism-invariant way. In the purely Riemannian (rather than conformal setting) these hypersurface invariants appear prominently as the prescribed data in initial value problems such as the Einstein field equations as expressed in the ADM formalism~\cite{ADM}. It should not be surprising that a complete understanding of these extrinsic curvatures is quite useful.

Given any defining function $s$ for the embedding $\Sigma \hookrightarrow (M,g)$, one can describe hypersurface invariants in terms of derivatives of $s$, the metric $g$, and their derivatives: for example, the unit conormal is given by $\hat{n} := \frac{ds}{|ds|_g}\big|_{\Sigma}$. Consequently, a useful organizing principle for hypersurface invariants is the notion of $g$-transverse order (defined in more detail in Section~\ref{riem-hyp-sec}). Loosely, viewing a defining function as a (local) coordinate that runs transverse to $\Sigma$, the $g$-transverse order of a hypersurface invariant is the integer $k$ where $\partial_s^k g$ is the leading transverse derivative term of $g$ in a coordinate representation of the invariant. In this regard, $\hat{n}$ has $g$-transverse order $0$ and, for a sufficiently generic embedding, $\II$ has $g$-transverse order $1$. In this article, $\II$ is viewed as a section of the symmetric square of the cotangent bundle on $\Sigma$.

Following the above in the Riemannian setting, one can characterize the invariants of a conformal hypersurface embedding $\Sigma \hookrightarrow (M,\cc)$ that respect the underlying conformal structure. In particular, given some metric $g \in \cc$ and an invariant $i^g$ of the Riemannian manifold $(M,g)$, if there exists some $w \in \mathbb{R}$ such that for any $\Omega$, we have $i^{\Omega^2 g} = \Omega^w i^g$ , then we say that such an invariant is \textit{conformally invariant} and has \textit{weight} $w$.

The task of characterizing such hypersurface invariants is unsurprisingly more subtle than in the Riemannian setting, as the Ricci calculus is manifestly diffeomorphism but not conformally invariant. Nonetheless, there exist several well-known examples, one of which is the trace-free part (with respect to the induced metric on $\Sigma$) of the second fundamental form, $\IIo := \II - H \bar{g}$; bars $\bar{\bullet}$ are used to represent objects induced on the hypersurface $\Sigma$. For $\Sigma \hookrightarrow (M,\cc)$ and $g,\Omega^2 g \in \cc$, we have that
$$\IIo^{\Omega^2 g} = \Omega \IIo^{g}\,,$$
where $\IIo := \II - H \bar{g}$ is the trace-free part of the second fundamental form for its associated embedding. Note that here and elsewhere in this article we implicitly use the restriction of $\Omega$ to the hypersurface $\Sigma$. In this article, we are interested in characterizing conformally invariant tensors that are hypersurface invariants; we call these tensors \textit{conformal hypersurface invariants}.


Analogous to the above Riemannian discussion, we can describe a conformal hypersurface embedding $\Sigma \hookrightarrow (M^d,\cc)$ by defining functions. In Lemma~\ref{to-hyper}, we prove that the $g$-transverse order of a representative of a conformal hypersurface invariant is compatible with the conformal structure, and thus introduce the notion of $\cc$-transverse order. Thus we can also organize conformal hypersurface invariants using this integer. For sufficiently general conformal hypersurface embeddings $\Sigma \hookrightarrow (M,\cc)$, $\IIo$ has $\cc$-transverse order $1$.

In addition to the trace-free second fundamental form, there is an entire family of trace-free symmetric rank-$2$ tensor-valued conformal hypersurface invariants with increasing $\cc$-transverse order and decreasing weight: the so-called \textit{conformal fundamental forms} $\{\IIo,\IIIo,\IVo, \dots\}$ as described in~\cite{Blitz1}. While in general, these tensors are challenging to compute, low-lying results are known~\cite{BGW1}:
\begin{align*}
\IIIo_{ab} =&\,  W_{\hat n ab \hat n} \\
\IVo_{ab} =&\, C_{\hat n(ab)}^\top + H W_{\hat n ab \hat n} + \tfrac{1}{d-5} \bar{\nabla}^c W_{c(ab) \hat{n}}^\top  \text{ for } d =4 \text{ or } d \geq6\,.
\end{align*}
A formula for $\Vo$ is lengthy but also given in~\cite{BGW1}. The superscript $\top$ denotes the projection of the tensor to the hypersurface (and $\otop$ denotes the trace-free equivalent), $W$ denotes the Weyl tensor, and $C$ denotes the Cotton tensor. (See Section~\ref{conventions} for more information on our conventions, and see Section~\ref{conf-hyp-sec} for a precise definition of a conformal fundamental form.) Many previous results in the mathematics literature can be expressed in terms of such tensors~\cite{Anderson,hypersurface_old,AGW,StChang,BGW2}, which suggests that they play a key role in the construction of conformal hypersurface invariants. Furthermore, in the physics literature such tensors have appeared as well~\cite{Kop,MarsPeon}. Indeed, one finds that for conformal hypersurface invariants with sufficiently small $\cc$-transverse order, the extrinsic information of the embedding $\Sigma \hookrightarrow (M,\cc)$ is captured entirely by the conformal fundamental forms. We formalize this observation in the main result of this article.
\begin{theorem} \label{main-result}
Let $\Sigma \hookrightarrow (M,\cc)$ with be a smoothly embedded conformal hypersurface and let $I$ be a natural conformal hypersurface invariant with $\cc$-transverse order $k \in \mathbb{Z}_{\geq 0}$. Then, if there exists a set of conformal fundamental forms $\{\IIo, \ldots, \FF{k+1}\}$ for $\Sigma \hookrightarrow (M,\cc)$, for any $g \in \cc$, there exists a formula for $I$ expressed as a (partial) contraction polynomial in elements of the set
$$\{\bar{g}, \bar{g}^{-1}, \hat{n}, \bar{\nabla}, \bar{R}, \IIo, \dots, \FF{k+1}\}\,.$$
\end{theorem}

\noindent
Loosely, natural conformal hypersurface invariants of $\Sigma \hookrightarrow (M,\cc)$ are those conformal hypersurface invariants whose representatives can be expressed in terms of a polynomial in a defining function, its derivatives, the metric $g$, its Levi-Civita connection, and its curvatures (all restricted to $\Sigma$)---see Definition~\ref{NHI} and Definition~\ref{NCHI} for more details, and see Section~\ref{conventions} for a description of (partial) contraction polynomials.

The proof of Theorem~\ref{main-result} is given in Section~\ref{conf-hyp-sec}. For a generic conformal hypersurface embedding, the existence of conformal fundamental forms is only well-understood in certain cases. In Section~\ref{conf-hyp-sec} we describe the state-of-the-art characterization of these conformally-invariant tensors and in Corollary~\ref{specialization} put concrete bounds on the constructability of the invariants arising in Theorem~\ref{main-result}.



\subsection{Conventions} \label{conventions}

In this section we detail our geometric conventions and notations. For the remainder of this article, we will assume that $M$ is a smooth manifold with dimension $d \geq 4$, and we will equip this manifold with either a metric $g$ or a conformal class of metrics $\cc$, both of which will be positive definite. Viewed as a Riemannian manifold, $(M,g)$ has a unique Levi-Civita connection $\nabla$ and curvature tensor
$$R(x,y)z = \nabla_x \nabla_y z - \nabla_y \nabla_x z - \nabla_{[x,y]} z\,,$$
where $x$, $y$, and $z$ are vector fields on $M$ and $[\pdot, \pdot]$ is the Lie bracket.

In this article, we will use Penrose's abstract index notation (with early Latin letters) to represent tensors and their contractions via the Einstein summation notation. For example, we can express the above equation for the Riemann curvature in the following abstract index notation:
$$x^a y^b R_{ab}{}^c{}_d z^d = x^a \nabla_a y^b \nabla_b z^c - y^a \nabla_a x^b \nabla_b z^c - [x^a (\nabla_a y^b) - y^a (\nabla_a x^b)] \nabla_b z^c\,.$$
In this notation, we will use round brackets to represent the symmetric part of a given tensor structure. That is, we will write $v_{a(bc)} := \tfrac{1}{2}(v_{abc} + v_{acb})$. As is typical, we will use the metric and its inverse to raise and lower indices.

We will sometimes replace the index of a tensor with the symbolic name of a (co)vector, so that if $u$ and $v$ are vectors at a point $p \in M$ and $T$ is a rank-$2$ covariant tensor at $p$, we can express $T(u,v) = u^a v^b T_{ab} \equiv T_{uv}$. The same notation will be used for the dual inner product associated with $g$, so that if $\alpha$ and $\beta$ are covectors at a point $p \in M$, then
$$T(g^{-1}(\alpha,\pdot), g^{-1}(\beta,\pdot)) \equiv g^{ab} g^{cd} \alpha_a \beta_c T_{bd} \equiv T_{\alpha \beta}\,.$$
We will also use the $|\pdot|_g$ notation to refer to the metric norm of a (co)vector, \textit{i.e.} for $v$ a vector and $\alpha$ a covector at a point $p \in M$, we say that $|v|^2_g= g(v,v)$ and $|\alpha|^2_g = g^{-1}(\alpha,\alpha)$, where $g$ is defined at $p$. By positive definiteness, we will always be able to take the square root to obtain $|v|_g$ and $|\alpha|_g$, respectively. Note that these notions extend in the obvious way to (co)vector and tensor fields.

A trace-correction of the Ricci tensor known as the Schouten tensor makes heavy appearance in this article. For the Ricci tensor given by the trace of the Riemann tensor $Ric_{ab} = R_{ca}{}^c{}_b$, the Schouten tensor is given by
$$P = \frac{1}{d-2} \left(Ric - \frac{Sc}{2(d-1)} g\right)\,,$$
where $Sc$ is the scalar curvature $Sc := Ric^a_a$, and we denote the trace of the Schouten tensor by $J := P_a^a$. Using this tensor, we can decompose the Riemann curvature into a trace-free part, the Weyl tensor $W$, and traceful parts:
$$R_{abcd} = W_{abcd} + g_{ac} P_{bd} - g_{bc} P_{ad} - g_{ad} P_{bc} + g_{bd} P_{ac}\,.$$
The divergence of the Weyl tensor in $d \geq 4$ yields the Cotton tensor:
$$C_{abc} = \tfrac{1}{d-3} \nabla^d W_{dcab}\,.$$

For a given Riemannian manifold $(M,g)$, we say that the metric $g$ is \textit{generic} when $g$ and its partial derivatives do not satisfy any identities beyond those that hold for every metric.

Given a collection of tensors with various tensor structures, we say that a \textit{contraction polynomial} in that collection of tensors is any linear combination of monomials formed by contracting the tensors in the family with each other in any way that forms a scalar. Examples using just the metric, its inverse, and the Riemann curvature include $\Sc$, $R_{abcd} R^{abcd}$, their sum, and more.
Note that the Levi-Civita connection can appear in such a family, in which case the order of terms in a monomial must be preserved and brackets must be used to specify upon which tensors the covariant derivative acts. For a \textit{partial contraction polynomial}, each monomial can be tensor-valued, however each monomial in a linear combination forming a given partial contraction polynomial must be of the same tensor type. Examples using just the metric, its inverse, and curvatures include $R_{abcd}$, $Ric_a{}^b Ric_{bc}$, and $ Ric_a{}^b Ric_{bc} +2  Sc \, Ric_{ac}$.

\section{Natural Hypersurface Invariants} \label{riem-hyp-sec}

We begin by providing a concrete definition of the hypersurface invariants described in the introduction as well as the notions of transverse order.
\begin{definition} \label{NHI}
Let $\Sigma \hookrightarrow (M,g)$ be a smoothly embedded hypersurface with $g$ generic and let $s$ be any defining function for $\Sigma$. Let $I[g,s]$ be the restriction to $\Sigma$ of a (partial) contraction polynomial in the set $\{s, |ds|^{-1}_g, g, g^{-1}, \nabla, R\}$. Then $I[g,s]$ is a \textit{natural hypersurface invariant} (NHI) when $I[g,s] = I[g,\tilde{s}]$ for any defining functions $s,\tilde{s}$ for $\Sigma$.

Furthermore, let $\mathsf{P}[g,s]$ be a polynomial in $\nabla$ with coefficients in (tensor-valued, when appropriate) NHIs. Then we say that $\mathsf{O}[g,s] := \Sigma \circ \mathsf{P}[g,s]$ is a \textit{natural hypersurface operator} (NHO), where the restriction operator $\Sigma$ restricts tensors on $M$ to $\Sigma$.
\end{definition}

%

Sometimes when it is clear from context, we will drop the metric dependence of an NHI. Furthermore, because for a given hypersurface embedding $\Sigma \hookrightarrow (M,g)$, an NHI is independent of choice of defining function $s$, we will also sometimes drop the dependence on the defining function when the particular embedding is specified.

Now, given a smooth hypersurface embedding $\Sigma \hookrightarrow (M,g)$ with $g$ generic and an arbitrary defining function $s$, we can treat $\Sigma$ as a subset of $M$ and thus consider a collar neighborhood $U := [0,\epsilon) \times \Sigma \subset M$ for some $0 < \epsilon$. On $U$ we can choose local coordinates $(s,y^i)$ for $i \in \{1, \ldots, d-1\}$ such that, on $\Sigma$, $\partial/\partial s$ is orthogonal to $\Sigma$ and each $\partial/\partial y^i$ is tangent to $\Sigma$.
In these coordinates, we see that $|ds|_g^{-1} = 1/\sqrt{g^{ss}}$, the components of the inverse metric, the Christoffel symbols, and the Riemann curvature are all expressible in terms of coordinate components of the metric and their partial derivatives. So, any scalar-valued NHI is expressible as a polynomial in $s$, the coordinate components of the metric and its inverse, $1/\sqrt{g^{ss}}$, and their partial derivatives. Similarly, the coordinate components of tensor-valued NHIs will be expressible in the same. However, these coordinate representations are not, in general, unique.

Given an NHI, we can study its properties by examining the set of its coordinate representations when evaluated on a generic metric. In any given coordinate representation, there exists a non-negative integer $k$ such that, along $\Sigma$, no metric components $g_{ab}$ are acted upon by more than $k$ normal derivatives $\partial_s^k$ (note that this includes derivatives of inverse metric components and derivatives of $1/\sqrt{g^{ss}}$ by the Leibniz rule). However, this integer $k$ is not an invariant of the NHI. As an example, consider the following coordinate representations for the same object: $\partial_s g_{ab}$ and $\partial_s g_{ab} + \frac{1}{2}(\partial_s^2 g_{ab} - \partial_s^2 g_{ba})$. Clearly, such coordinate representations yield same the same object but the maximal derivative countings are distinct. But because it is bounded below, we can minimize over all such coordinate representations to obtain an invariant of a given NHI. We define the \textit{transverse order} of an NHI as the minimal $k$ over all such coordinate representations when evaluated on a generic metric. Often, we will discuss NHIs evaluated on a particular (non-generic) metric $g$; in that case, we define the \textit{$g$-transverse order} of an NHI as the minimal $k$ over all such coordinate representations when evaluated on the metric $g$. Note that the $g$-transverse order of an NHI is at most equal to its transverse order.

We can also study various orders of NHOs. Given a conformal hypersurface embedding $\Sigma \hookrightarrow (M,g)$ with generic $g$ and such an operator $\mathsf{O}[g,s]$, we say that it has \textit{transverse differential order} $k \in \mathbb{Z}_{\geq 0}$ when there exists $v$ in the domain of $\mathsf{O}$ such that $\mathsf{O}(s^{k} v) \neq 0$ but, for every $v'$ in the domain of $\mathsf{O}$, we have that $\mathsf{O}(s^{k+1} v') = 0$. We also say that $\mathsf{O}$ has \textit{transverse coefficient order} $\ell \in \mathbb{Z}_{\geq 0}$ when, amongst coefficient NHIs of $\mathsf{O}$, the maximal such transverse order is $\ell$. As before, when we are interested in the transverse orders of such an operator when evaluated on a  particular metric $g$, we analogously define \textit{$g$-transverse differential order} and \textit{$g$-transverse coefficient order}. As before, note that these transverse orders are at most the transverse differential order and the transverse coefficient order, respectively.

\medskip

While the given definition of an NHI is useful for considering transverse order, it is somewhat impractical for calculational purposes. Fortunately, any NHI is expressible using a more familiar set of tensors, as in the following proposition of Gover and Waldron~\cite{Will1}.

\begin{proposition} \label{easy-set}
Let $\Sigma \hookrightarrow (M^d,g)$ be a smoothly embedded hypersurface, let $s$ be any defining function for $\Sigma$, and let $I[g,s]$ be an NHI for $\Sigma \hookrightarrow (M,g)$. Further, we define  $\hat{n} := \frac{ds}{|ds|_g}\big|_{\Sigma}$. Then, $I[g,s]$ can be expressed as the restriction to $\Sigma$ of a (partial) contraction polynomial in
$$\{g\,, g^{-1}\,,\hat{n}\,, \II, \bar{\nabla} \II, \ldots, \bar{\nabla}^\ell \II, R, \nabla R, \ldots, \nabla^m R\}\,,$$
for some finite non-negative integers $\ell$ and $m$.
\end{proposition}
\begin{proof}
First, observe that $I[g,s]$ is independent of the choice of defining function $s$. From~\cite[Section 3.3]{Wald}, in a neighborhood of $\Sigma$ there always exists a defining function $\tilde{s}$ for $\Sigma$ such that $|d \tilde{s}|_g^2 = 1$. Thus, we have that $I[g,s] = I[g,\tilde{s}]$, and so in particular we can write $I[g,s]$ as the restriction to $\Sigma$ of a (partial) contraction polynomial in $\{\tilde{s}, g, g^{-1}, \nabla, R\}$. However, Gover and Waldron~\cite[Proposition 2.7]{Will1} show that if $|d\tilde{s}|^2_g = 1$, then for any integer $0 \leq k$ we have that $\nabla^k \tilde{s}|_{\Sigma}$ can be expressed as the restriction to $\Sigma$ of a partial contraction polynomial in
$$\{g\,, g^{-1}\,,\hat{\tilde{n}}\,, \II, \bar{\nabla} \II, \ldots, \bar{\nabla}^\ell \II, R, \nabla R, \ldots, \nabla^m R\}$$
for some finite non-negative integers $\ell$ and $m$. So, $I[g,\tilde{s}]$---and hence $I[g,s]$---can be expressed as such a (partial) contraction polynomial.  Finally, note that $\hat{n} = \hat{\tilde{n}}$ and that $\II$ is independent of the choice of defining function. Thus, we can express $I[g,s]$ as the restriction to $\Sigma$ of a (partial) contraction polynomial in 
$$\{g\,, g^{-1}\,,\hat{n}\,, \II, \bar{\nabla} \II, \ldots, \bar{\nabla}^\ell \II, R, \nabla R, \ldots, \nabla^m R\}\,.$$
\end{proof}
%

We are now positioned to analyze these Riemannian invariants with a finer-toothed comb. To do so, we first define a useful equivalence relation on NHIs that characterize them by their leading transverse derivative orders.

\begin{definition}
Let $\Sigma \hookrightarrow (M,g)$ be a smoothly embedded hypersurface, fix $k \in \mathbb{Z}_{\geq 1}$, and let $A$ and $B$ be two NHIs of the same tensor type. We
define (for $k \geq 1$) the equivalence relation~$\tsim{k}$ on NHIs by equating $A \tsim{k} B$ when $A - B$ has $g$-transverse order at most $k-1$.
%
\end{definition}

With this definition in hand, we can examine the transverse order of various projections of normal derivatives of the Riemann curvature tensor:
\begin{lemma} \label{reduce-to}
Let $\Sigma \hookrightarrow (M,g)$ be a smoothly embedded hypersurface. Then, for for every integer $m \geq 0$, the following equivalences hold:
\begin{align*}
\top \colon \nabla^m_{\hat n} \colon R_{abcd}  &\tsim{m+2} 0 \\
\top \hat{n}^d \colon \nabla^m_{\hat n} \colon R_{dabc} &\tsim{m+2} 0 \\
\top \hat{n}^c \hat{n}^d \colon \nabla^m_{\hat n} \colon R_{cabd} &\tsim{m+2} -(d-2) \otop \colon \nabla^m_{\hat n} \colon P_{ab} - \bar{g}_{ab} \colon \nabla^m_{\hat n} \colon J \\
\top \colon \nabla^m_{\hat n} \colon Ric_{ab} &\tsim{m+2}  (d-2) \otop \colon \nabla^m_{\hat n} \colon P_{ab} + \bar{g}_{ab} \colon \nabla^m_{\hat n} \colon J \\
\top \hat{n}^b \colon \nabla^m_{\hat n} \colon Ric_{ab} &\tsim{m+2} 0 \\
\hat{n}^a \hat{n}^b \colon \nabla^m_{\hat n} \colon Ric_{ab} &\tsim{m+2} (d-1) \colon \nabla^m_{\hat n} \colon J \\
\colon \nabla^m_{\hat n} \colon Sc|_{\Sigma} &\tsim{m+2} 2(d-1) \colon \nabla^m_{\hat n} \colon J|_{\Sigma}\,.
\end{align*}
In the above, we denote $\colon \nabla^m_{\hat n} \colon := \hat{n}^{a_1} \cdots \hat{n}^{a_m} \nabla_{a_1} \cdots \nabla_{a_m}$.
\end{lemma}

\begin{proof}
The cases where $m = 0$ follow from the definition of the Schouten tensor, the fact that $\II$ has transverse order $1$, and standard hypersurface identities:
\begin{align} 
R_{abcd}^\top &\eqSig \bar{R}_{abcd} - \II_{ac} \II_{bd} + \II_{ad} \II_{bc}\,, \label{gauss} \\
R_{abc \hat n}^\top &\eqSig \bar{\nabla}_a \II_{bc} - \bar{\nabla}_b \II_{ac}\,,\label{cod}\\
Sc - Ric_{\hat n \hat n}  &\eqSig\bar{Sc} + \II^2 - (\operatorname{tr} \II)^2\,, \label{theorema} \\
W_{\hat n ab \hat n} + (d-3) P^\top_{ab} &\eqSig \IIo_{ab}^2 - \tfrac{1}{2(d-2)} \IIo^2 \bar{g}_{ab} + (d-3)\left(\bar{P}_{ab} + H \IIo_{ab} + \tfrac{1}{2} H^2 g_{ab} \right)\,. \label{fialkow}
\end{align}
We write $\eqSig$ to represent identities that hold along $\Sigma$. Equation~\nn{gauss} is the Gau\ss\ equation, Equation~\nn{cod} is the Codazzi--Mainardi equation, Equation~\nn{theorema} is Gau\ss' \textit{theorema egregium}, and Equation~\nn{fialkow} is the Fialkow--Gau\ss\ equation~\cite{Will1}.

To show these equivalences hold for $m \geq 1$, note that $\hat{n}_a \nabla_{\hat n} \eqSig \nabla_a - \nabla^\top_a$, where $\nabla^\top$ is an NHO with transverse coefficient and derivative orders $0$. Then the equivalences follow from rearranging derivatives (which leads to lower order terms) and application of the Bianchi identities. Finally, we can express Ricci and scalar curvatures in terms of the Schouten tensor and its trace.
\end{proof}

Using this lemma, we provide a set of distinguished generating tensors for hypersurface invariants.

\begin{theorem} \label{riem-invs-general}
Let $\Sigma \hookrightarrow (M^d,g)$ with $d$ even be a smoothly embedded hypersurface with $g$ generic and let $I[g]$ be an NHI for $\Sigma \hookrightarrow (M,g)$ with $g$-transverse order $k \in \mathbb{Z}_{\geq 0}$. Then, if $k < 2$, $I[g]$ can be expressed as a (partial) contraction polynomial in
$$\{\bar{g}, \bar{g}^{-1},\bar{\nabla}, \bar{R}, \hat{n}, \IIo, H\}\,,$$
and if $k \geq 2$, then $I[g]$ can be expressed as a (partial) contraction polynomial in
$$\{\bar{g}, \bar{g}^{-1},\bar{\nabla}, \bar{R}, \hat{n}, \IIo, \otop \colon \nabla_{\hat n} \colon P, \ldots, \otop \colon \nabla_{\hat n}^{k-2} \colon P, H, J|_{\Sigma},  \ldots, \colon \nabla^{k-2}_{\hat n} \colon J|_{\Sigma}\}\,.$$
\end{theorem}

\begin{proof}
Let $I[g]$ be an NHI with transverse order $k \in \mathbb{Z} \cap [0,d-2]$. Using Lemma~\ref{easy-set}, we can express $I[g]$ as the restriction to $\Sigma$ of a (partial) contraction polynomial in
$$\{g\,, g^{-1}\,,\hat{n}\,, \II, \bar{\nabla} \II, \ldots, \bar{\nabla}^\ell \II, R, \nabla R, \ldots, \nabla^m R\}\,.$$
Note that $\II$ has transverse order $1$, and all projections of derivatives of Riemann tensors involve at least one factor of $\II$, so if $k = 0$, by Weyl's classical invariant theory, $I[g]$ is expressable as a (partial) contraction polynomial in $\{\bar{g}, \bar{g}^{-1}, \bar{\nabla}, \bar{R}, \hat{n}\}$.

If $k=1$, then $I[g]$ must be expressable in terms of
$$\{\bar{g}, \bar{g}^{-1}, \bar{\nabla}, \bar{R}, \hat{n}, \II\}\,.$$
 This follows from the Gau\ss\ equation, the Codazzi--Mainardi equation, and Proposition~\ref{easy-set}. But the definition of $\IIo$ then implies that $I[g]$ can be expressed as a (partial) contraction polynomial in
$$\{\bar{g}, \bar{g}^{-1},\bar{\nabla}, \bar{R}, \hat{n}, \IIo, H\}\,.$$

Now consider the case where $k \geq 2$. Note that $\nabla = \nabla^\top + \hat{n} \nabla_{\hat n}$, so the following operator identity holds:
$$\nabla_{a_1} \cdots \nabla_{a_m} \eqSig \hat{n}_{a_1} \cdots \hat{n}_{a_m} \colon \nabla_{\hat n} \colon + \text{lower transverse order operators.}$$
So without loss of generality it is sufficient to consider only terms appearing in $I[g]$ of the form $\colon \nabla_{\hat n}^{k-2} \colon R|_{\Sigma}$. In particular, from Lemma~\ref{reduce-to}, it is clear that if $I[g]$ has $g$-transverse order $k$, then it follows that $I[g]$ is a (partial) contraction polynomial in
$$\{\bar{g}, \bar{g}^{-1},\bar{\nabla}, \bar{R}, \hat{n}, \IIo, \otop \colon \nabla_{\hat n} \colon P, \ldots, \otop \colon \nabla_{\hat n}^{m-2} \colon P, H, J|_{\Sigma},  \ldots, \colon \nabla^m_{\hat n} \colon J|_{\Sigma}\}\,.$$
This completes the proof.
\end{proof}

In general, non-genericity of $g$ in $\Sigma \hookrightarrow (M^d,\cc)$ can only reduce the structures that NHIs can depend on in the above theorem. That is, Theorem~\ref{riem-invs-general} also holds for non-generic $g$, although non-genericity may result in some linear dependence amongst elements of the generating set.

\bigskip

\color{black}

\section{Natural Conformal Hypersurface Invariants} \label{conf-hyp-sec}
We are now ready to examine the structure of conformal hypersurface invariants, using the tools developed above. We first provide a definition.
\begin{definition} \label{NCHI}
Let $\Sigma \hookrightarrow (M,\cc)$ be a conformal hypersurface embedding with defining function $s$ and, for any $g \in \cc$, let $I[g,s]$ be an NHI. Then, we say that $I[\cc,s]$ is a \textit{natural conformal hypersurface invariant} (NCHI) of weight $w$ when, for any $\Omega \in C^\infty_+ M$, we have that $I[\Omega^2 g,s] = \Omega^w I[g, s]$ for some $w \in \mathbb{R}$.
\end{definition}

As before, when $\cc$ and $s$ can be inferred from context, we will drop their dependencies. With this definition in mind, we will show that transverse order is a well-defined notion for an NCHI.
\begin{lemma} \label{to-hyper}
Let $\Sigma \hookrightarrow (M^d,\cc)$ be a conformal hypersurface embedding with defining function $s$ and let $I[\cc,s]$ be an NCHI of weight $w$. Further,  let $g,\tilde{g} \in \cc$. If $I[g,s]$ has $g$-transverse order $k$, then $I[\tilde{g},s]$ has $\tilde{g}$-transverse order $k$. 
\end{lemma}

\begin{proof}
Fixing $\Sigma \hookrightarrow (M,\cc)$, let $I[\cc,s]$ be an NCHI of weight $w$ and let $I[g,s]$ be a representative. On a collar neighborhood of $\Sigma$, choose local coordinates $(s,y^i)$. We now consider a coordinate representation of $I[g,s]$ where the highest normal derivative term is $k$, which exists by the definition of $g$-transverse order. Then, in this coordinate representation,
$$
I[g,s] = \Big[\mathsf{O}[g,s]\left(\partial_s^k g_{ab}\right) + \mathsf{P}[g,s](g_{ab}) \Big] \,,
$$
where $\mathsf{O}[g,s]$ is a non-zero NHO with $g$-transverse differential order $0$ and $g$-transverse coefficient order at most $k$, and $\mathsf{P}[g,s]$ is an NHO with $g$-transverse differential order at most $k-1$ and $g$-transverse coefficient order at most $k-1$.  Also, $\partial_s g_{ab}$ refers to the matrix of normal derivatives of metric coordinate components. Note that no coefficient term in $\mathsf{O}[g,s]$ with $k$ normal derivatives can completely cancel the $\partial_s^k g_{ab}$ term on which $\mathsf{O}[g,s]$ acts, otherwise this coordinate representation would not achieve a highest normal derivative count of $k$.

By definition, because $g,\tilde{g} \in \cc$, there exists $\Omega \in C^\infty_+ M$ such that $\tilde{g} = \Omega^2 g$. So, we now compute:
\begin{align*}
I[\tilde{g},s] =& \Omega^w I[g, s] \\
=& \Omega^w \Big\{ \mathsf{O}[g,s] \left( \partial_{s}^k (g_{ab}) \right) + \mathsf{P}[g,s](g_{ab})\big\} \\
=& \Omega^w \Big\{ \mathsf{O}[g,s] \left( \partial_{s}^k (\Omega^{-2} \tilde{g}_{ab}) \right) + \mathsf{P}[g,s](\Omega^{-2} \tilde{g}_{ab})\big\} \\
=& \Omega^w \Big\{ \hat{\mathsf{O}}[\tilde{g},{s}] \left( \partial_{{s}}^k (\Omega^{-2} \tilde{g}_{ab}) \right) + \hat{\mathsf{P}}[\tilde{g},{s}](\Omega^{-2} \tilde{g}_{ab})\big\}\,.
\end{align*}
The second line follows from the definition of $g$-transverse order $k$. To obtain the fourth equality, consider the first term on the third line: the coordinate components of $g$ only appear in the coefficients of the operator, and hence can be replaced with $\Omega^{-2} \tilde{g}$, making each of these coefficients a functional of $\tilde{g}$---we denote this new operator by $\hat{O}[\tilde{g},{s}]$. Such terms are of the form $\partial^\ell (\Omega^{-2} \tilde{g})$, and hence can be written as the sum of $\Omega^{-2} \partial^{\ell} \tilde{g}$ and terms containing fewer derivatives of $\tilde{g}$. So, the $\tilde{g}$-transverse coefficient order of $\hat{\mathsf{O}}[\tilde{g},{s}]$ matches the $g$-transverse coefficient order of $\mathsf{O}[g,{s}]$. Similarly, the same holds to rewrite $\mathsf{P}[g,{s}]$ in terms of $\hat{\mathsf{P}}[\tilde{g},{s}]$. Note that this procedure cannot produce an operator with $\tilde{g}$-transverse derivative order larger than the $g$-transverse derivative order of the corresponding unhatted operator, and similarly for the $\tilde{g}$-transverse coefficient order.

Now observe that there are two possible sources of terms with $k$ normal derivatives acting on $\tilde{g}$ in $I[\tilde{g},{s}]$: namely, the single term arising from $\partial_{{s}}^k \Omega^{-2} \tilde{g}_{ab}$, and any possible coefficients in $\hat{\mathsf{O}}[\tilde{g},{s}]$ that also contain $k$ normal derivatives of components of $\tilde{g}$. However, the only way these terms maintain their number of normal derivatives upon converting from $g$ to $\tilde{g}$ is to pass the factors of $\Omega$ through the partial derivatives. But then the resulting terms are proportional to those found in $I[g,s]$, where the factor of proportionality is just a power of $\Omega$. But by the argument above, we know that cancellation of these terms does not occur in $I[g,s]$ and thus cannot occur in $I[\tilde{g},{s}]$. So, in this coordinate representation, there is at least one term with $k$ normal derivatives on $\tilde{g}$, and no terms with more than $k$.

Now suppose that there is a coordinate representation where $I[\tilde{g},{s}]$ has fewer than $k$ normal derivatives. Then by the reverse of the calculation and argument detailed above, this implies that there exists a corresponding coordinate representation in which $I[g,s]$ also have fewer than $k$ normal derivatives, a contradiction. Thus, $I[\tilde{g},{s}]$ has $\tilde{g}$-transverse order $k$.
\end{proof}
In general, for $\Sigma \hookrightarrow (M,g)$ with defining function $s$, the $g$-transverse order of any particular NHI is not invariant under conformal rescalings: this only occurs when said NHI is a representative of an NCHI. As an example, consider the mean curvature $H^g$ of $\Sigma$. A short calculation shows that for generic $g$, $H^g$ has $g$-transverse order $1$. However, there always exists a function $\Omega \in C^\infty_+ M$ such that $H^{\Omega^2 g} = 0$, which (clearly) has $\Omega^2 g$-transverse order $0$. On the other hand, Lemma~\ref{to-hyper} allows us to extend the notion of transverse order to NCHIs. For a conformal hypersurface embedding $\Sigma \hookrightarrow (M,\cc)$ with defining function $s$, we say that an NCHI has \textit{$\cc$-transverse order} $k$ when, for any $g \in \cc$, we have that $I[g,s]$ has $g$-transverse order $k$. 

\medskip

To prove Theorem~\ref{main-result}, we must now turn our attention to conformal fundamental forms. Let $\Sigma \hookrightarrow (M,g)$. According to~\cite{BGW1}, an $m^{\text{th}}$ \textit{conformal fundamental form} is any natural (in the sense described above) trace-free section $L$ of symmetric squares of the cotangent bundle of $\Sigma$ with transverse order $m-1$ that satisfies
$$L^{\Omega^2 g} = \Omega^{3-m} L^g\,,$$
for any $\Omega \in C^\infty_+ M$. Various constructions are provided for such tensors. While the constructions for $m^{\text{th}}$ fundamental forms differ slightly in~\cite{BGW1} and~\cite{Blitz1}, the core construction is the same. In particular, when $3 \leq m < \frac{d+3}{2}$ or $d$ is even and $3 \leq m \leq d-1$, we have that
$$\FF{m} \tsim{m-1} \alpha_m \otop \colon \nabla_{\hat n}^{m-2} \colon \IIo^{\rm e} \,,$$
where $\alpha_m$ is some non-zero coefficient, $\IIo^{\rm e} := \nabla_{(a} \nabla_{b)\circ} s + s \mathring{P}_{ab}$, and $s$ is any defining function for $\Sigma \hookrightarrow (M,g)$ that obeys
$$|ds|^2_g - \tfrac{2}{d}(s \Delta s + Js^2) = 1 + \mathcal{O}(s^{d})\,.$$
Note that this condition is conformally invariant. In particular, using Lemma~\ref{reduce-to}, we have that
$$\FF{m} \tsim{m-1} \beta_m \otop \colon \nabla_{\hat n}^{m-3} P\,,$$
where again $\beta_m$ is a non-zero coefficient.

While we do not have canonical constructions for a $d^{\text{th}}$ fundamental form or higher in general, by the definition given, one can directly construct examples of $\IVo$ and $\Vo$ when $d=4$. A canonical construction for $\IVo$ in the case where $d=4$ is given in~\cite[Section 3]{Blitz1}, where one finds that
$$\IVo \tsim{3} \otop \nabla_{\hat n} P\,.$$
Similarly, in~\cite[Section 1]{BGW1}, a construction for $\Vo$ is given in $d=4$, with
$$\Vo \tsim{4} \otop \Delta P \tsim{4} \otop \colon \nabla_{\hat n}^2 \colon P\,.$$
These are all examples of a corollary of Proposition~\ref{easy-set}.

\begin{corollary} \label{all-FFs}
Let $\Sigma \hookrightarrow (M,g)$ be a smoothly embedded hypersurface. For any integers $m \in \mathbb{Z}_{\geq 3}$, if there exists an $m^{\text{th}}$ conformal fundamental form $\FF{m}$, then
$$\FF{m} \tsim{m-1} \alpha \otop \colon \nabla_{\hat n}^{m-3} \colon P\,,$$
for some non-zero constant $\alpha$.
\end{corollary}
\begin{proof}
Using the set from Proposition~\ref{easy-set}, we attempt to construct $\FF{m}$ as per its definition: a rank-2 tensor with conformal weight $3-m$ and transverse order $m-1$. Candidate terms are of the form
$$(g^{-1})^{a_1} \bar{\nabla}^{a_2} (\II)^{a_3} \hat{n}^{a_4} (\colon \nabla_{\hat n}^{a_5} \colon) R^{a_6}\,,$$
where each $a_i \geq 0$. In particular, to satisfy the transverse order condition, we require that $a_5 \geq m-3$ and $a_6 \geq 1$. To satisfy the weight condition, we must have that
$$-2 a_1 + a_3 + a_4 - a_5 + 2 a_6 = 3-m\,,$$
and to ensure this candidate has the correct tensor structure, we must have that
$$-2a_1 + a_2 + 2 a_3 + a_4 + 4 a_6 = 2\,.$$
Because each $a_i$ is an integer and $\hat{n}$ can contract into $R$ at most twice, we find that there are two unique solutions to this system of equations:
$$\colon \nabla_{\hat n}^{m-3} \colon Ric_{ab} \qquad \text{and} \qquad \hat{n}^c \hat{n}^d \colon \nabla_{\hat n}^{m-3} \colon R_{cabd}\,.$$
However, from Lemma~\ref{reduce-to}, these both reduce to the desired tensor structure as in the corollary, modulo trace terms. 
Because conformal fundamental forms are hypersurface trace-free, the corollary follows.
\end{proof}

Establishing this fact is essential to proving Theorem~\ref{main-result}, because it allows us to apply Theorem~\ref{riem-invs-general}.

%

\begin{proof}[Proof of Theorem~\ref{main-result}]
From Lemma~\ref{to-hyper}, because $I[\cc]$ has $\cc$-transverse order $k$, for every $g \in \cc$ we have that $I[g]$ has $g$-transverse order $k$. Now, suppose that conformal fundamental forms $\{\IIo, \ldots, \FF{k+1}\}$. We now look to show that for every $g \in \cc$, we have that $I[g]$ can be expressed as a (partial) contraction polynomial in
 $$\{\bar{g}, \bar{g}^{-1},\bar{\nabla}, \bar{R}, \hat{n}, \IIo, \ldots, \FF{k+1}\}\,.$$

We prove the theorem in the stronger case where we assume $\cc$ is generic. Let $\Sigma \hookrightarrow (M^d,\cc)$ with $d$ even be a smoothly embedded conformal hypersurface and let $I[\cc]$ be an NCHI of this embedding. Clearly, if $I[\cc]$ has $\cc$-transverse order $0$, then in any metric representative,  it can be formed from elements of the set $\{\bar{g}, \bar{g}^{-1}, \bar{\nabla}, \bar{R}\}$ and the unit conormal $\hat{n}$.

Now consider the case where $k \geq 2$. We wish to show that for every $g \in \cc$, $I[g]$ can be expressed as a (partial) contraction polynomial in elements of
$$\{\bar{g}, \bar{g}^{-1},\bar{\nabla}, \bar{R}, \hat{n}, \IIo, \ldots, \FF{k+1}\}\,.$$
From Theorem~\ref{riem-invs-general}, we see that $I[g]$ must be expressable as a (partial) contraction polynomial in
$$\{\bar{g}, \bar{g}^{-1},\bar{\nabla}, \bar{R}, \hat{n}, \IIo, \otop \colon \nabla_{\hat n} \colon P, \ldots, \otop \colon \nabla_{\hat n}^{k-2} \colon P, H, J|_{\Sigma},  \ldots, \colon \nabla^{k-2}_{\hat n} \colon J|_{\Sigma}\}\,.$$
From Corollary~\ref{all-FFs}, because we have assumed that $\{\IIo, \ldots, \FF{k+1}\}$ exists, we can replace instances of normal derivatives of the Schouten tensor with their corresponding conformal fundamental forms, leaving us with the generating set
$$\{\bar{g}, \bar{g}^{-1},\bar{\nabla}, \bar{R}, \hat{n}, \IIo, \ldots, \FF{k+1}, H, J|_{\Sigma},  \ldots, \colon \nabla^{k-2}_{\hat n} \colon J|_{\Sigma}\}\,.$$
It remains to eliminate $\{H, J|_{\Sigma}, \ldots, \colon \nabla_{\hat n}^{k-2} \colon J|_{\Sigma}\}$ from the generating set; we do so by contradiction.

Suppose that for some representative $g \in \cc$, in the polynomial expression for $I[g]$ there exists at least one summand containing at least one factor of $\colon \nabla_{\hat n}^m \colon J|_{\Sigma}$. Then, without loss of generality, there exists a maximal integer $m$ such that $\colon \nabla_{\hat n}^{j} \colon J|_{\Sigma}$ is not contained in any summand of $I[g]$ for every $j > m$. Furthermore, there exists an integer $\ell \geq 1$ such that no summand contains more than $\ell$ factors of $\colon \nabla_{\hat n}^m \colon J|_{\Sigma}$.

Now because $I[g]$ is a representative of an NCHI, it must be conformally invariant. However, observe that$$
\colon (\nabla^{\Omega^2 g})_{\hat{n}}^{m} \colon J^{\Omega^2 g}|_{\Sigma}= \colon (\nabla^{\Omega^2 g})_{\hat{n}}^{m} \colon [\Omega^{-2} (J^g - \nabla^g \cdot \Upsilon + (1 - \tfrac{d}{2}) \Upsilon^2)]|_{\Sigma} \,,$$
where $\Upsilon = d \log \Omega$. Notably, expanding the above expression in terms of $\nabla^g$, the above expression has terms containing $-\Omega^{-m-2} \colon (\nabla^g)_{\hat n}^{m} \colon \nabla^g \cdot \Upsilon|_{\Sigma}$. Thus, in terms of invariants of $\Sigma \hookrightarrow (M,g)$, we have that there exists at least one summand in the polynomial expression for $I[\Omega^2 g]$ containing exactly $\ell$ factors of $- \Omega^{-m} \colon (\nabla^g)_{\hat n}^{m-2} \colon \nabla^g \cdot \Upsilon|_{\Sigma}$. But no lower order terms can cancel these conformal variations, so by the conformal invariance of $I[\cc]$, the coefficients of such summands must cancel internally for every $\Omega$, and hence must vanish for $\Omega = 1$, \textit{i.e.} for the initial choice of $g \in \cc$. This is a contradiction, as we assumed that $\colon \nabla_{\hat n}^m \colon J|_{\Sigma}$ appears at least once in $I[g]$.

The case where $m=1$ is identical to the above argument, except we have that
$$H^{\Omega^2 g} = \Omega^{-1} (H^g + \hat{n} \cdot \Upsilon)\,,$$
leading to the same contradiction. This completes the proof.
\end{proof}

Known results regarding existence of certain conformal fundamental forms lead to specializations of Theorem~\ref{main-result} that are themselves interesting. In particular, we have the following corollary.
\begin{corollary} \label{specialization}
Let $\Sigma \hookrightarrow (M^d,\cc)$ be a smooth conformal hypersurface embedding and let $I[\cc]$ be an NCHI of this embedding with $\cc$-transverse order $k \geq 0$. Then, $I[\cc]$ can be expressed as a (partial) contraction polynomial in
$$\{\bar{g}, \bar{g}^{-1}, \hat{n}, \bar{\nabla}, \bar{R}, \IIo, \dots, \FF{k+1}\}\,,$$
for $k \leq \frac{d-1}{2}$ and $d \geq 5$ odd, for $k \leq d-2$ and $d \geq 6$ even, and for $k \leq 4$ for $d =4$.
\end{corollary}
\begin{proof}
In the case where $d \geq 5$ odd~\cite{BGW1} provides a construction for $m^{\text{th}}$ conformal fundamental forms for every $m  \leq \frac{d+1}{2}$. This reference also provides a construction for $\{\IIo, \IIIo, \IVo, \Vo\}$ in the case where $d =4$. In the case where $d \geq 6$ even~\cite{Blitz1} provides a construction for $m^{\text{th}}$ conformal fundamental forms for every $m \leq d-1$. The corollary then follows directly from Theorem~\ref{main-result}.
\end{proof}

\begin{remark}
Of particular interest in the $d=4$ case is that $\Vo \tsim{4} \otop B$, where $B$ is the Bach tensor.
\end{remark}

\section{Conclusion}
By providing a complete characterization of natural conformal hypersurface invariants up to a certain transverse derivative order, Theorem~\ref{main-result} and Corollary~\ref{specialization} give a useful generating set for conformally-invariant constraints on the conformal infinities of asymptotically de Sitter or anti-de Sitter spacetimes. An example of this has already been provided in the asymptotically de Sitter case~\cite{GoKop}, where in $d=4$, the conformal fundamental forms appeared as constraints on the asymptotic stress-energy tensor of the spacetime. We have shown that such relationships are generic and generalize both to arbitrary dimension and also to asymptotically anti-de Sitter spacetimes, thus proving that the conformal fundamental forms are indeed essential to understanding the relationship between the asymptotic features of a spacetime and the geometry of the corresponding conformal infinity.

While the work presented in this article only concerns spacelike and timelike hypersurfaces, the methods derived herein can likely be generalized to classify conformal invariants of null hypersurfaces as well. There is already significant interest in the curvature invariants that detect null boundaries such as black hole event horizons~\cite{FaraoniNielsen,McNuttPage,McNutt}, so a complete classification of conformal null hypersurface invariants would provide clarity in this endeavor.

\section*{Acknowledgements}
The author would like to thank A. Gover, D. McNutt, J. \v{S}ilhan, and A. Waldron for helpful comments and discussions. The author further gratefully acknowledges the support by the Czech Science Foundation (GACR) via grant GA22-00091S.

\bibliographystyle{siam}
\bibliography{bib}

\end{document}